\documentclass[12pt]{amsart}
\usepackage{amssymb}
 \theoremstyle{plain}
\newtheorem{theorem}{Theorem}
\newtheorem{corollary}{Corollary}

\newtheorem{lemma}{Lemma}
\newtheorem{proposition}{Proposition}

\theoremstyle{definition}

\theoremstyle{remark}

\numberwithin{equation}{section}

\setbox0=\hbox{$+$}

\newdimen\plusheight
\plusheight=\ht0
\def\+{\;\lower\plusheight\hbox{$+$}\;}

\setbox0=\hbox{$-$}
\newdimen\minusheight
\minusheight=\ht0
\def\-{\;\lower\minusheight\hbox{$-$}\;}

\setbox0=\hbox{$\cdots$}
\newdimen\cdotsheight
\cdotsheight=\plusheight
\def\cds{\lower\cdotsheight\hbox{$\cdots$}}
\begin{document}
\baselineskip=17pt
\title[ Some more Long Continued Fractions, I ]
       { Some more Long Continued Fractions, I}

\author{James Mc Laughlin}
\address{Mathematics Department\\
 Anderson Hall\\
West Chester University, PA 19383} \email{jmclaughl@wcupa.edu}

\author{Peter Zimmer}
\address{Mathematics Department\\
 Anderson Hall\\
West Chester University,  PA 19383} \email{pzimmer@wcupa.edu}

\thanks{The authors would like to thank the referee of an earlier
version of this paper for pointing out a number of relevant
papers, of which they were previously unaware.}
\date{\today}
\begin{abstract}
In this paper we show how to construct several infinite families
of polynomials $D(\bar{x},k)$, such that $\sqrt{D(\bar{x},k)}$ has
a regular continued fraction expansion with arbitrarily long
period, the length of this period being controlled by the positive
integer parameter $k$.

We also describe how to quickly compute the fundamental units in the
corresponding real quadratic fields.

\end{abstract}

\maketitle

\section{Introduction}
Let  $D(x_1,x_2,\dots, x_r, k)$ be a polynomial in
$\mathbb{Z}[x_1,x_2,\dots, x_r]$, where $x_1$, $x_2$,$\dots, x_r$
are integer variables and $k$ is a positive integer parameter
which appears as an exponent in the expression of
$D(x_1,x_2,\dots, x_r, k)$. Suppose further that there are
positive lower bounds $T_1$ and $T_2$ such that, for all integral
$x_i\geq T_1$ and all integral $k\geq T_2$, the surd
$\sqrt{D(x_1,x_2,\dots, x_r, k)}$ has a regular continued fraction
of the form
\[
\sqrt{D(x_1,x_2,\dots, x_r, k)}=[\,a_{0}; \overline{a_{1}, \cdots
, a_{n_k},2a_{0}}], \] where each $a_{j}:=a_{j}(x_{1},\cdots,
x_{r},k)\in \mathbb{Z}[x_{1},\cdots, x_{r}]$, for $j=0,1,\dots
n_k$, and the length of the period, $n_k+1$, depends only on $k$
($k$ may also be present in some of the $a_j$ as an exponent).
Under these circumstances we say that $D(x_1,x_2,\dots, x_r, k)$
has a \emph{long continued fraction expansion} and call the
expansion $[\,a_{0}; \overline{a_{1}, \cdots , a_{n_k},2a_{0}}]$ a
\emph{long continued fraction}.

We give the following example due to Madden \cite{M01} to
illustrate the concept:\\
Let $b$,  $n$ and $k$ be positive integers. Set
\[
D:=D(b,n,k):=
   {\left(  b\,{\left( 1 + 2\,b\,n \right) }^k
        +n \right)  }^2 + 2\,{\left( 1 + 2\,b\,n \right) }^k.
\]
Then  {\allowdisplaybreaks
\begin{align*}
\sqrt{D}=\bigg[b&{\left( 1 + 2bn \right) }^k +n;\overline{ b,
  2b{\left( 1 + 2bn \right) }^{k-1},}\overline{b(2bn+1),
  2b{(2bn+1)}^{k-2},}\\&
\overline{ b(2bn+1)^2,
  2b{(2bn+1)}^{k-3},\,\,\dots,\,\,}
%\\&\phantom{sadasdsadfsa}
%\overline{\dots}
%\\&
\overline{b{(1 + 2bn)}^{k-1},2b,}\\
&\overline{n+b{(2bn+1)}^k,}
\\
&\overline{2b,b{(1 + 2bn)}^{k-1},\,\,\dots,\,\,}
%\\&\phantom{sadasdsadfsa}\dots,\\&
\overline{2b{(2bn+1)}^{k-3},b(2bn+1)^2,}\\
&\overline{2b{(2bn+1)}^{k-2},b(2bn+1),}
%\\&
\overline{2b{\left( 1 + 2bn \right) }^{k-1},b,}
%\\&
\overline{2( n + b{\left( 1 + 2bn \right) }^k)} \bigg ].
\end{align*}
}The fundamental period in the continued fraction expansion has
length $4k+2$.

Since the discovery of Daniel Shanks \cite{S69} and \cite{S71},
there have been a number of examples of families of quadratic
surds whose continued fraction expansions have unbounded period
length. These include those discovered by Hendy \cite{H74},
Bernstein \cite{B76} and \cite{B76a}, Williams \cite{W85},
\cite{W95}  and \cite{W00} , Levesque and
 Rhin \cite{LR86}, Azuhatu \cite{A87}, Levesque \cite{L88},
Halter-Koch \cite{H89a} and \cite{H89b}, Mollin and Williams
\cite{MW92a}  and \cite{MW92b},
 van der Poorten
 \cite{VDP94}, and, more recently, Madden \cite{M01} and Mollin
 \cite{M03}.

 Williams paper \cite{W85} contains several tables
 listing surds with long continued fraction expansions, along with
 the length of their fundamental period. In Mollin and Williams
\cite{MW92a} and \cite{MW92b} and Williams \cite{W95} and
\cite{W00}, the authors describe a general method for computing
the fundamental period $\pi (N)$ of  the continued fraction
expansion of $\sqrt{N}$, where
\[
N= (\sigma(q r a^n+\mu(a^k+\lambda)/q)/2)^2-\sigma^2 \mu \lambda
a^n r.\] Here $\mu, \lambda \in \{-1,1\}$, $\sigma \in \{1,2\}$
(the value of $\sigma$ depends on the parity of $q r
a^n+\mu(a^k+\lambda)/q$), $q r|a^k+\lambda$, $\gcd (n,k)=1$, and
$n>k\geq 1$.
 To describe this method we need some additional notation.
 For positive integers $r>1$ and $s$, let the regular
continued fraction expansion of
\[s/r=[q_0;q_1,\dots,q_m],\]
 with $q_m>1$. Define
$M(r,s):=2\lfloor (m+1)/2 \rfloor $. For a fixed $Q$ and $a$ with
$\gcd (a,Q) =1$, denote by $s_i$ the integer satisfying
\[
s_i \equiv a^i (\mod Q),
\]
with $1 \leq s_i <Q$. Let $\omega = \omega (a,Q)$ denote the
multiplicative order of $a$ modulo $Q$ and define \[
W(a,q)=\sum_{i=1}^{\omega}M(s_i,Q).
\]
The authors give various formulae for $\pi(N)$, formulae that depend
on the various parameters and the functions $W(\cdot,\cdot)$ and
$\omega(\cdot,\cdot)$.

\vspace{10pt}

 In this present paper we use matrix methods (based on Madden's method
in \cite{M01}) to
 derive several  infinite families of  quadratic surds with
 long continued fraction expansions.
 We feel that, while these matrix methods may not be as widely
 applicable, where they can be used
 the proofs are more transparent, more direct and less intricate
 than proofs by the method outlined above in the papers of Mollin
 and Williams.
Furthermore, some of the families of long continued fractions in
the present paper  generalize some of the results in
 some of the papers cited previously in this introduction.

 \section{preliminaries}
We first recall some basic properties of continued fractions. For
any sequence of numbers $a_{0}, a_{1},
 \dots$, define, for $i\geq 0$, the numbers $A_{i}$ and $B_{i}$
 (the $i$-th numerator convergent and $i$-th denominator convergent,
 respectively, of the continued fraction below)
 by
 \[
a_{0} \+ \frac{1}{a_{1}} \+ \frac{1}{a_{2}} \+ \cds \frac{1}{a_{i}
} = \frac{A_{i}}{B_{i}}.
\]
By the correspondence between matrices and continued fractions, we
have that
\begin{equation}\label{mateq2}
\left(
\begin{matrix}
a_{0} & 1 \\
1     & 0
\end{matrix}
\right) \left(
\begin{matrix}
a_{1} & 1 \\
1     & 0
\end{matrix}
\right) \dots \left(
\begin{matrix}
a_{i} & 1 \\
1     & 0
\end{matrix}
\right)  =\left(
\begin{matrix}
A_{i} & A_{i-1} \\
B_{i}     & B_{i-1}
\end{matrix}
\right).
\end{equation}
It is also well known that
\begin{align}\label{recur}
A_{i} &= a_{i}A_{i-1} + A_{i-2}, \\
B_{i} &= a_{i}B_{i-1} + B_{i-2}.\notag
\end{align}
See \cite{LW92}, for example, for these basic properties of
continued fractions.

 Our starting point is the following elementary result.
 \begin{lemma}\label{l1}
 Let $q_{0}, q_{1}, \dots, q_{1}, q_{0}$ be a finite palindromic sequence
 of positive integers (with or without a central term) and let
 \begin{equation}\label{mateq}
 \left(
\begin{matrix}
q_{0} & 1 \\
1     & 0
\end{matrix}
\right) \left(
\begin{matrix}
q_{1} & 1 \\
1     & 0
\end{matrix}
\right) \dots \left(
\begin{matrix}
q_{1} & 1 \\
1     & 0
\end{matrix}
\right) \left(
\begin{matrix}
q_{0} & 1 \\
1     & 0
\end{matrix}
\right) =:\left(
\begin{matrix}
w & u \\
u     & v
\end{matrix}
\right).
 \end{equation}
 Then
 \begin{equation}
 \sqrt{w/v}=[q_{0}; \overline{q_{1},q_{2}, \dots q_{2},
 q_{1},2q_{0}}].
 \end{equation}
 \end{lemma}
 \begin{proof}
 As usual, $\overline{q_{1},q_{2}, \dots q_{2},
 q_{1},2q_{0}}$ means that the sequence of partial quotients
 $q_{1},q_{2}, \dots q_{2},
 q_{1},2q_{0}$ is repeated infinitely often.
 Note  that the matrix on the right side of \eqref{mateq} is
 symmetric, since the left side is a symmetric product of
 symmetric matrices.

Let $\alpha =[q_{0}; \overline{q_{1},q_{2}, \dots q_{2},
 q_{1},2q_{0}}]$, so that
 $\alpha =[q_{0}; q_{1},q_{2}, \dots q_{2},
 q_{1},q_{0}+\alpha]$.
Then, by \eqref{mateq}, \eqref{mateq2} and \eqref{recur}, we have
that
\begin{align*}
\alpha = \frac{w+ \alpha u}{u+\alpha v},
\end{align*}
so that $\alpha^{2} = w/v$ and the result follows.
\end{proof}
We will be primarily interested in the case where $w/v$ is an
integer. The next result, although entirely trivial, is also
central to what follows.

\begin{lemma}\label{l2}
Let $D$ be any complex number and let $\alpha$ and $\beta$ be any
complex numbers such that $\alpha+\beta \not = 0$. Define the
matrix $P$ by
\[
P=\left(
\begin{matrix}
-\sqrt{D}& \sqrt{D} \\
1     & 1
\end{matrix}
\right)\left(
\begin{matrix}
\alpha & 0 \\
0    & \beta
\end{matrix}
\right)\left(
\begin{matrix}
-\sqrt{D}& 1 \\
\sqrt{D}    & 1
\end{matrix}
\right).
\]
Then
\[
\frac{P_{1,1}}{P_{2,2}} = D.
\]
\end{lemma}

To investigate the fundamental units in the corresponding real
quadratic fields $\mathbb{Q}(\sqrt{D})$, there is the following
theorem on page 119 of~\cite{N90}:
\begin{theorem}\label{narth}
Let $D$ be a square-free, positive rational integer and let
$K=\mathbb{Q}(\sqrt{D})$. Denote by $\epsilon_{0}$ the fundamental
unit of $K$ which exceeds unity, by $s$ the period of the
continued fraction expansion for $\sqrt{D}$, and by $P/Q$ the
($s-1$)-th approximant of it.

If $D \not \equiv 1 \mod{4}$ or $D \equiv 1 \mod{8}$, then
\[\epsilon_{0} = P + Q \sqrt{D}.
\]
However, if $D \equiv 5 \mod{8}$, then
\[\epsilon_{0} = P + Q \sqrt{D}.
\]
or
\[\epsilon_{0}^{3} = P + Q \sqrt{D}.
\]
Finally, the norm of $ \epsilon_{0}$ is positive if the period $s$
is even and negative otherwise.
\end{theorem}

This theorem implies the following result.
\begin{proposition}\label{pfu}
Let $D$ be a non-square positive integer, $D \not \equiv 5
(\mod{8})$. Suppose $\sqrt{D}=[q_0;\overline{q_1,\dots, q_1,2
q_0}]$, and that
\begin{multline*}
 \left(
\begin{matrix}
q_{0} & 1 \\
1     & 0
\end{matrix}
\right) \left(
\begin{matrix}
q_{1} & 1 \\
1     & 0
\end{matrix}
\right) \dots \left(
\begin{matrix}
q_{1} & 1 \\
1     & 0
\end{matrix}
\right) \left(
\begin{matrix}
q_{0} & 1 \\
1     & 0
\end{matrix}
\right) =\left(
\begin{matrix}
w & u \\
u     & v
\end{matrix}
\right)\\ = \left(
\begin{matrix}
-\sqrt{D}& \sqrt{D} \\
1     & 1
\end{matrix}
\right)\left(
\begin{matrix}
\alpha & 0 \\
0    & \beta
\end{matrix}
\right)\left(
\begin{matrix}
-\sqrt{D}& 1 \\
\sqrt{D}    & 1
\end{matrix}
\right).
 \end{multline*}
Then the fundamental unit in $\mathbb{Q}(\sqrt{D})$ is $2 \sqrt{D}
\beta$.
\end{proposition}
\begin{proof}
By Theorem \ref{narth} and \eqref{mateq2}, the fundamental unit in
$\mathbb{Q}(\sqrt{D})$ is $u + v \sqrt{D}$, and the equality of
this quantity and
 $2 \sqrt{D} \beta$ follows from comparing corresponding entries
 in the matrices above.
\end{proof}

Remark: Other approaches can be used to calculate the fundamental
unit in the case $D \equiv 5 (\mod{8})$, but we do not pursue that
here.

 We will follow Madden and let $\vec{N}$ denote the
sequence $a_1, a_2, \dots, a_j$ whenever
\[
N=\left(
\begin{matrix}
a_{1} & 1 \\
1     & 0
\end{matrix}
\right) \dots \left(
\begin{matrix}
a_{j} & 1 \\
1     & 0
\end{matrix}
\right),
\]
and let $\overset{\leftarrow}{N}$ denote the sequence $a_j,
a_{j-1}, \dots, a_2, a_1$.

\section{Some General Propositions}
 We next state several variants of a result of Madden from
\cite{M01}. These general propositions will allow us to  construct
specific families of long continued fractions in the next section.

\begin{proposition}\label{p1}
Let $k$, $u$, $v$, $w$ and $r$ be positive integers such that
$rw/v$ is an integer, and let  $x$ be a rational such that $w x$
and $2 v x$ are  integers. Let
\[ C= \left(
\begin{matrix}
r & 0 \\
0    & 1
\end{matrix}
\right).
\]
Suppose further, for each  integer $n \in \{0,1,\dots k-1\}$, that
the matrix $N_n$ defined by
\[
N_n:= C^{-n}\left(
\begin{matrix}
u & r^{k-1}v \\
w     & r u -2 v w x
\end{matrix}
\right)C^n =\left(
\begin{matrix}
u & r^{k-1-n}v \\
r^n w     & r u -2 v w x
\end{matrix}
\right)
\]
has an expansion of the form
\begin{equation}\label{Neq}
N_n=\left(
\begin{matrix}
a_{1}^{(n)} & 1 \\
1     & 0
\end{matrix}
\right)\left(
\begin{matrix}
a_{2}^{(n)} & 1 \\
1     & 0
\end{matrix}
\right) \dots \left(
\begin{matrix}
a_{j_{n}}^{(n)} & 1 \\
1     & 0
\end{matrix}
\right),
\end{equation}
where each $a_{i}^{(n)}$ is a positive integer. Then
\begin{multline}\label{sqrteq}
\sqrt{\frac{w r^k}{v} +w^2x^2} \\= [w x; \overline{
\vec{N}_{k-1}\vec{N}_{k-2}\dots \vec{N}_{1}\vec{N}_{0}, 2 v
x,\overset{\leftarrow}{N}_{0},\overset{\leftarrow}{N}_{1},\dots,
\overset{\leftarrow}{N}_{k-2},\overset{\leftarrow}{N}_{k-1}, 2 w
x}].
\end{multline}

If $w r^k/v+w^2 x^2 \not \equiv 5 (\mod{8})$ and is square free,
then the fundamental unit in $\mathbb{Q}(w r^k/v+w^2 x^2)$ is
\[
\frac{w
%2\,{\sqrt{\frac{w\,
%          \left( r^k + v\,w\,x^2 \right) }{v}}
 %         }\,
    {\left( r\,u + v\,\left( - w\,x   +
           {\sqrt{\frac{w\,
                 \left( r^k + v\,w\,x^2 \right) }{v}}}
           \right)  \right) }^{2\,k}}{{v\left( -
          w\,x  +
       {\sqrt{\frac{w\,\left( r^k + v\,w\,x^2 \right) }
           {v}}} \right) }^2}.
\]
\end{proposition}
\begin{proof}
Let $D= w r^k/v +w^2x^2$. We consider the matrix product
\[
\left(
\begin{matrix}
w x & 1 \\
1     & 0
\end{matrix}
\right)N_{k-1} \dots N_1 N_0\left(
\begin{matrix}
2 v x & 1 \\
1     & 0
\end{matrix}
\right)N_0^T N_1^T \dots N_{k-1}^T\left(
\begin{matrix}
w x & 1 \\
1     & 0
\end{matrix}
\right).
\]
By the definition of the $N_n$, this product equals
\[
\left(
\begin{matrix}
w x & 1 \\
1     & 0
\end{matrix}
\right)C^{-k}(CN_{0})^k\left(
\begin{matrix}
2 v x & 1 \\
1     & 0
\end{matrix}
\right)((CN_{0})^k)^{T}C^{-k}\left(
\begin{matrix}
w x & 1 \\
1     & 0
\end{matrix}
\right).
\]
Define the matrix $M$ by
\[
M:=\left(
\begin{matrix}
v x -  \frac{v\sqrt{D}
    }{w}&  v x + \frac{ v\sqrt{D}
    }{w}\\
        \phantom{a}&\phantom{b}\\
1     & 1
\end{matrix}
\right).
\]
One can check that {\allowdisplaybreaks
\begin{align*}
&\left(
\begin{matrix}
w x & 1 \\
1     & 0
\end{matrix}
\right)C^{-k} = \left(
\begin{matrix}
-\sqrt{D}& \sqrt{D} \\
1     & 1
\end{matrix}
\right)\left(
\begin{matrix}
\frac{1}{-w x-\sqrt{D}}& 0 \\
0    &\frac{1}{-w x+\sqrt{D}}
\end{matrix}
\right) M^{-1};\\
&\phantom{as}\\
 &(CN_{0})^k = M \left(
\begin{matrix}
ru-v w x-v\sqrt{D}&  0\\
        \phantom{a}&\phantom{b}\\
0    & ru-v w x+v\sqrt{D}
\end{matrix}
\right)^k M^{-1}; \\
&\phantom{as}\\
&M^{-1}\left(
\begin{matrix}
2 v x & 1 \\
1     & 0
\end{matrix}
\right)(M^{-1})^{T}=\left(
\begin{matrix}
\frac{-w}{2 v \sqrt{D}}&  0\\
        \phantom{a}&\phantom{b}\\
0    &\frac{w}{2 v \sqrt{D}}
\end{matrix}
\right).
\end{align*}
}
 Thus {\allowdisplaybreaks\begin{multline*} \left(
\begin{matrix}
w x & 1 \\
1     & 0
\end{matrix}
\right)C^{-k}(CN_{0})^k\left(
\begin{matrix}
2 v x & 1 \\
1     & 0
\end{matrix}
\right)((CN_{0})^k)^{T}C^{-k}\left(
\begin{matrix}
w x & 1 \\
1     & 0
\end{matrix}
\right)=\\
\left(
\begin{matrix}
-\sqrt{D}& \sqrt{D} \\
1     & 1
\end{matrix}
\right)\phantom{asdsfddsfsdfdfsfsdfsadfasdsafasdfsadfdfsfddsfdf}\\
\times\left(
\begin{matrix}
\frac{1}{-w x-\sqrt{D}}& 0 \\
0    &\frac{1}{-w x+\sqrt{D}}
\end{matrix}
\right)\left(
\begin{matrix}
ru-v w x-v\sqrt{D}&  0\\
        \phantom{a}&\phantom{b}\\
0    & ru-v w x+v\sqrt{D}
\end{matrix}
\right)^k \\
\times\left(
\begin{matrix}
\frac{-w}{2 v \sqrt{D}}&  0\\
        \phantom{a}&\phantom{b}\\
0    &\frac{w}{2 v \sqrt{D}}
\end{matrix}
\right)\\
\times \left(
\begin{matrix}
ru-v w x-v\sqrt{D}&  0\\
        \phantom{a}&\phantom{b}\\
0    & ru-v w x+v\sqrt{D}
\end{matrix}
\right)^k \left(
\begin{matrix}
\frac{1}{-w x-\sqrt{D}}& 0 \\
0    &\frac{1}{-w x+\sqrt{D}}
\end{matrix}
\right) \\
\times \left(
\begin{matrix}
-\sqrt{D}& 1 \\
\sqrt{D}    & 1
\end{matrix}
\right).
\end{multline*}
}

 The result follows by Lemmas \ref{l1}, \ref{l2} and Proposition \ref{narth}.
\end{proof}

The fundamental period of the continued fractions above contain a
central partial quotient, namely $2vx$. We next show how to
construct long continued fractions which do not have a central
partial quotient.

\begin{proposition}\label{p2}
Let $u$, $v$, $x$ and $r$ be positive integers. Let \[ C= \left(
\begin{matrix}
r & 0 \\
0    & 1
\end{matrix}
\right).
\]
Suppose further, for each integer $k\geq 1$ and each integer $n
\in \{0,1,\dots k-1\}$, that the matrix $N_n$ defined by
\[
N_n:= C^{-n}\left(
\begin{matrix}
u & r^{k-1}v \\
r^k v     & r u -2 v  x
\end{matrix}
\right)C^n =\left(
\begin{matrix}
u & r^{k-1-n}v \\
r^{k+n}v     & r u -2 v  x
\end{matrix}
\right)
\]
has an expansion of the form
\begin{equation}\label{Neq2}
N_n=\left(
\begin{matrix}
a_{1}^{(n)} & 1 \\
1     & 0
\end{matrix}
\right)\left(
\begin{matrix}
a_{2}^{(n)} & 1 \\
1     & 0
\end{matrix}
\right) \dots \left(
\begin{matrix}
a_{j_{n}}^{(n)} & 1 \\
1     & 0
\end{matrix}
\right),
\end{equation}
where each $a_{i}^{(n)}$ is a positive integer. Then, for each
integer $k \geq 1$,
\begin{equation}\label{sqrteq2}
\sqrt{ r^{2k} +x^2} = [ x; \overline{
\vec{N}_{k-1}\vec{N}_{k-2}\dots \vec{N}_{1}\vec{N}_{0},
\overset{\leftarrow}{N}_{0},\overset{\leftarrow}{N}_{1},\dots,
\overset{\leftarrow}{N}_{k-2},\overset{\leftarrow}{N}_{k-1}, 2
x}].
\end{equation}
If $r^{2k} + x^2 \not \equiv 5 (\mod{8})$ and is square free, then
the fundamental unit in $\mathbb{Q}(\sqrt{r^{2k} +x^2})$ is
\[
\left( x + {\sqrt{r^{2\,k} + x^2}} \right) \,
  {\left( u + v\,\left( -x +
         {\sqrt{r^{2\,k} + x^2}} \right)/r  \right) }^
   {2\,k}.
   \]
\end{proposition}

\begin{proof}
We proceed as in the proof of Proposition \ref{p1}. Let $D =r^{2k}
+ x^2$. As above, the definition of $N_n$ implies
\begin{multline*}
\left(
\begin{matrix}
 x & 1 \\
1     & 0
\end{matrix}
\right)N_{k-1} \dots N_1 N_0 N_0^T N_1^T \dots N_{k-1}^T\left(
\begin{matrix}
x & 1 \\
1     & 0
\end{matrix}
\right)\\
=
 \left(
\begin{matrix}
x & 1 \\
1     & 0
\end{matrix}
\right)C^{-k}(CN_{0})^k((CN_{0})^k)^{T}C^{-k}\left(
\begin{matrix}
 x & 1 \\
1     & 0
\end{matrix}
\right). \end{multline*}
 Define the matrix $M$ by
\[
M:=\left(
\begin{matrix}
 \frac{x-\sqrt{D}
    }{r^k}&  \frac{ x+\sqrt{D}
    }{r^k}\\
        \phantom{a}&\phantom{b}\\
1     & 1
\end{matrix}
\right).
\]
One can check that {\allowdisplaybreaks
\begin{align*}
&\left(
\begin{matrix}
x & 1 \\
1     & 0
\end{matrix}
\right)C^{-k} = \left(
\begin{matrix}
-\sqrt{D}& \sqrt{D} \\
1     & 1
\end{matrix}
\right)\left(
\begin{matrix}
\frac{x-\sqrt{D}}{r^{2k}}& 0 \\
0    &\frac{x+\sqrt{D}}{r^{2k}}
\end{matrix}
\right) M^{-1};\\
&\phantom{as}\\
 &(CN_{0})^k = M \left(
\begin{matrix}
ru-v  x-v\sqrt{D}&  0\\
        \phantom{a}&\phantom{b}\\
0    & ru-v  x+v\sqrt{D}
\end{matrix}
\right)^k M^{-1}; \\
&\phantom{as}\\
&M^{-1}(M^{-1})^{T}=\left(
\begin{matrix}
\frac{1}{2}+\frac{x}{2 \sqrt{D}}&  0\\
        \phantom{a}&\phantom{b}\\
0    &\frac{1}{2}-\frac{x}{2 \sqrt{D}}
\end{matrix}
\right).
\end{align*}
} Thus {\allowdisplaybreaks\begin{multline*} \left(
\begin{matrix}
x & 1 \\
1     & 0
\end{matrix}
\right)C^{-k}(CN_{0})^k((CN_{0})^k)^{T}C^{-k}\left(
\begin{matrix}
 x & 1 \\
1     & 0
\end{matrix}
\right)=\\
\left(
\begin{matrix}
-\sqrt{D}& \sqrt{D} \\
1     & 1
\end{matrix}
\right)\phantom{asdsfddsfsdfdfsfsdfsadfasdsafasdfsadfdfsfddsfdf}\\
\times \left(
\begin{matrix}
\frac{x-\sqrt{D}}{r^{2k}}& 0 \\
0    &\frac{x+\sqrt{D}}{r^{2k}}
\end{matrix}
\right)\times \left(
\begin{matrix}
r u-v  x-v\sqrt{D}&  0\\
        \phantom{a}&\phantom{b}\\
0    & r u-v  x+v\sqrt{D}
\end{matrix}
\right)^k \\
\times\left(
\begin{matrix}
\frac{1}{2}+\frac{x}{2 \sqrt{D}}&  0\\
        \phantom{a}&\phantom{b}\\
0    &\frac{1}{2}-\frac{x}{2 \sqrt{D}}
\end{matrix}
\right)
\\
\times \left(
\begin{matrix}
ru-v  x-v\sqrt{D}&  0\\
        \phantom{a}&\phantom{b}\\
0    & ru-v  x+v\sqrt{D}
\end{matrix}
\right)^k \left(
\begin{matrix}
\frac{x-\sqrt{D}}{r^{2k}}& 0 \\
0    &\frac{x+\sqrt{D}}{r^{2k}}
\end{matrix}
\right) \\
\times \left(
\begin{matrix}
-\sqrt{D}& 1 \\
\sqrt{D}    & 1
\end{matrix}
\right),
\end{multline*}
} and once again the result follows by Lemmas \ref{l1} and
\ref{l2} and Proposition \ref{narth}.
\end{proof}

It is also possible to create long continued fractions with no
central partial quotient, but with two extra central partial
quotients that do not come from $\overset{\rightarrow}{N}_{0}$ and
$\overset{\leftarrow}{N}_{0}$.

\begin{proposition}\label{p3}
Let $u$, $v$, $w$, $q$ and $r$ be positive integers.  Let
\[ C= \left(
\begin{matrix}
r & 0 \\
0     & 1
\end{matrix}
\right).
\]
Suppose further, for each integer $k\geq 1$ and each integer $n
\in \{0,1,\dots k-1\}$, that the matrix $N_n$ defined by
{\allowdisplaybreaks \begin{align*} N_n&:= C^{-n}\left(
\begin{matrix}
u & r^{k-1}v \\
v\,\left( r^k + 4\,q\,w \right)    & r\,u - 2\,q\,r^k\,v -
    2\,\left( 1 + 4\,q^2 \right) \,v\,w
\end{matrix}
\right)C^n \\&=\left(
\begin{matrix}
u & r^{k-1-n}v \\
r^{n}v\,\left( r^k + 4\,q\,w \right)    & r\,u - 2\,q\,r^k\,v -
    2\,\left( 1 + 4\,q^2 \right) \,v\,w
\end{matrix}
\right)
\end{align*}
} has an expansion of the form
\begin{equation}\label{Neq3}
N_n=\left(
\begin{matrix}
a_{1}^{(n)} & 1 \\
1     & 0
\end{matrix}
\right)\left(
\begin{matrix}
a_{2}^{(n)} & 1 \\
1     & 0
\end{matrix}
\right) \dots \left(
\begin{matrix}
a_{j_{n}}^{(n)} & 1 \\
1     & 0
\end{matrix}
\right),
\end{equation}
where each $a_{i}^{(n)}$ is a positive integer. Then, for each
integer $k \geq 1$,
\begin{multline}\label{sqrteq3}
\sqrt{ r^k\,\left( r^k + 4\,q\,w \right)  +
  {\left( w + q\,\left( r^k + 4\,q\,w \right)  \right)
      }^2}=\\
  \bigg [w + q(r^k+4qw );\phantom{sdasdasdasdasdasasdasdasdasddaddasdadasdasdasdsad}\\
   \overline{ \vec{N}_{k-1}\vec{N}_{k-2}\dots
\vec{N}_{1}\vec{N}_{0},2q,2q,
\overset{\leftarrow}{N}_{0},\overset{\leftarrow}{N}_{1},\dots,
\overset{\leftarrow}{N}_{k-2},\overset{\leftarrow}{N}_{k-1},
2(w+q(r^k+4qw)) } \bigg].
\end{multline}
For ease of notation, set $D= r^k\,\left( r^k + 4\,q\,w \right)  +
  {\left( w + q\,\left( r^k + 4\,q\,w \right)  \right)
      }^2$ and $\gamma= q(r^k+4 q w)$. Then the fundamental unit
      in $\mathbb{Q}(\sqrt{D})$ is
      \[
\frac{{\left( u + v\,\left( {\sqrt{d}} - w -
           \gamma  \right)/r  \right) }^{2\,k}\,
    \left( {\sqrt{d}} - w + \gamma  \right) \,
    {\left( {\sqrt{d}} + w + \gamma  \right) }^2}{
      {\gamma }^2}.
      \]
\end{proposition}

\begin{proof}
We proceed as in the proof of Propositions \ref{p1} and \ref{p2}.
Let \[D =  r^k\,\left( r^k + 4\,q\,w \right)  +
  {\left( w + q\,\left( r^k + 4\,q\,w \right)  \right)
      }^2.
   \] As above, the definition of $N_n$ implies
{\allowdisplaybreaks
\begin{multline*} \left(
\begin{matrix}
 w + q\,\left( r^k + 4\,q\,w \right) & 1 \\
1     & 0
\end{matrix}
\right)N_{k-1} \dots N_1 N_0 \left(
\begin{matrix}
2q  & 1 \\
1     & 0
\end{matrix}
\right)\\
\times\left(
\begin{matrix}
 2q & 1 \\
1     & 0
\end{matrix}
\right)N_0^T N_1^T \dots N_{k-1}^T\left(
\begin{matrix}
w + q\,\left( r^k + 4\,q\,w \right) & 1 \\
1     & 0
\end{matrix}
\right)\\
=
 \left(
\begin{matrix}
w + q\,\left( r^k + 4\,q\,w \right) & 1 \\
1     & 0
\end{matrix}
\right)C^{-k}(CN_{0})^k\left(
\begin{matrix}
 2q & 1 \\
1     & 0
\end{matrix}
\right)\phantom{asdasdasdasdasdasd}\\
\times\left(
\begin{matrix}
 2q & 1 \\
1     & 0
\end{matrix}
\right)((CN_{0})^k)^{T}C^{-k}\left(
\begin{matrix}
 w + q\,\left( r^k + 4\,q\,w \right) & 1 \\
1     & 0
\end{matrix}
\right). \end{multline*} }
 Recall that $\gamma =
q(r^k+4 q w)$ and define the matrix $M$ by
\[
M:=\left(
\begin{matrix}
\displaystyle{ \frac{q \left(v (w+\gamma )-\sqrt{D} v\right)}{v
   \gamma }} &  \displaystyle{\frac{q \left((w+\gamma ) v+\sqrt{D}
   v\right)}{v \gamma }}\\
        \phantom{a}&\phantom{b}\\
1     & 1
\end{matrix}
\right).
\]
One can check (preferably with a computer algebra system like
\emph{Mathematica}) that {\allowdisplaybreaks
\begin{align*}&
\left(
\begin{matrix}
 w + q\,\left( r^k + 4\,q\,w \right)& 1 \\
1     & 0
\end{matrix}
\right)C^{-k}\\
&\phantom{as} = \left(
\begin{matrix}
-\sqrt{D}& \sqrt{D} \\
1     & 1
\end{matrix}
\right)\left(
\begin{matrix}
\displaystyle{\frac{q r^{-k} \left(w+\gamma
   -\sqrt{D}\right)}{\gamma }} & 0 \\
0    &\displaystyle{\frac{q r^{-k} \left(w+\gamma
   +\sqrt{D}\right)}{\gamma }}
\end{matrix}
\right) M^{-1};\\
&\phantom{as}\\
 &(CN_{0})^k = M \left(
\begin{matrix}
r u-v \left(w+\gamma +\sqrt{D}\right)&  0\\
        \phantom{a}&\phantom{b}\\
0    &r u+v \left(-w-\gamma +\sqrt{D}\right)
\end{matrix}
\right)^k M^{-1}; \\
&\phantom{as}\\
&M^{-1}\left(
\begin{matrix}
 2q & 1 \\
1     & 0
\end{matrix}
\right)\left(
\begin{matrix}
 2q & 1 \\
1     & 0
\end{matrix}
\right)(M^{-1})^{T}=\left(
\begin{matrix}
\displaystyle{ \frac{1}{2}+\frac{w-\gamma}{2 \sqrt{D}}}&  0\\
       \phantom{a}&\phantom{b}\\
0    &\displaystyle{\frac{1}{2}-\frac{w-\gamma}{2 \sqrt{D}}}
\end{matrix}
\right).
\end{align*}
} Thus, as in Propositions \ref{p1} and \ref{p2}, the result
follows by Lemmas \ref{l1} and \ref{l2} and Proposition
\ref{narth}.
\end{proof}

\section{long continued fractions}
We next give specific  values (in terms of other variables) for
some of the variables in the propositions so as to produce
explicit long continued fractions. The main problem is to do this
in such a way that the matrices $N_n$ satisfy \eqref{Neq},
\eqref{Neq2} or \eqref{Neq3}.

From Proposition \ref{p1}, we have that
\[
N_0 =\left(
\begin{matrix}
u & r^{k-1}v \\
 w     & r u -2 v w x
\end{matrix}
\right),
\]
so one obvious approach is to initially define $u$, $v$, $w$, $x$
and $r$ from a product of the form
\[
N_{0}= \left(
\begin{matrix}
u & r^{k-1}v \\
 w     & r u -2 v w x
\end{matrix}
\right)= \left(
\begin{matrix}
a_{1}^{(0)} & 1 \\
1     & 0
\end{matrix}
\right)\left(
\begin{matrix}
a_{2}^{(0)} & 1 \\
1     & 0
\end{matrix}
\right) \dots \left(
\begin{matrix}
a_{j_{0}}^{(0)} & 1 \\
1     & 0
\end{matrix}
\right)
\]
and then specialize the $a_{i}^{(0)}$ so that each $N_{n}$ also
has an expansion of a similar form. We proceed similarly with
Propositions \ref{p2} and \ref{p3}.

We consider one example in detail to illustrate the method. We
consider Proposition \ref{p1} with
\[
N_{0}= \left(
\begin{matrix}
u & r^{k-1}v \\
 w     & r u -2 v w x
\end{matrix}
\right)= \left(
\begin{matrix}
a & 1 \\
1     & 0
\end{matrix}
\right)\left(
\begin{matrix}
b & 1 \\
1     & 0
\end{matrix}
\right)=\left(
\begin{matrix}
1+a b & a \\
b     & 1
\end{matrix}
\right).
\]
Upon comparing $(1,2)$ entries, it can be seen that $a$ must be a
multiple of $r^{k-1}$, so replace $a$ by $r^{k-1}a$ and we then
have
\[
\left(
\begin{matrix}
u & r^{k-1}v \\
 w     & r u -2 v w x
\end{matrix}
\right)= \left(
\begin{matrix}
1+r^{k-1}a b & r^{k-1}a \\
b     & 1
\end{matrix}
\right).
\]
Then $v=a$, $w=b$, $u=1+ a b r^{k-1}$ and $x=(a b r^{k}+r-1)/(2
ab)$. The requirements that $w x$, $w r/v$ and $2 v x$ be integers
force $b$ and $r$ to have the forms $b = 2 m a$, $r=1+2 a m s$,
respectively,  for some integers $m$ and $s$. Thus  we finally have
{\allowdisplaybreaks
\begin{align*}
 u&=1 + 2\,a^2\,m\,{\left( 1 + 2\,a\,m\,s \right) }^{-1 + k},\\
  v&=a,\\
w&=  2\,a\,m, \\
x&=\frac{s + a\,{\left( 1 + 2\,a\,m\,s \right) }^k}
   {2\,a}.
   \end{align*}
}

With these values, we have that {\allowdisplaybreaks
\begin{align*}
\frac{w r^k}{v} + w^2 x^2&=m\,\left( 2\,{\left( 1 + 2\,a\,m\,s
\right) }^k +
    m\,{\left( s + a\,{\left( 1 + 2\,a\,m\,s \right) }^k
         \right) }^2 \right),\\
N_n&=\left(
\begin{matrix}
u & r^{k-1-n}v \\
 w r^n    & r u -2 v w x
\end{matrix}
\right)\\&= \left(
\begin{matrix}
1 + 2\,a^2\,m\,{\left( 1 + 2\,a\,m\,s \right) }^{-1 + k}
& a\,{\left( 1 + 2\,a\,m\,s \right) }^{-1 + k - n} \\
 2\,a\,m\,{\left( 1 + 2\,a\,m\,s \right) }^n   &1
\end{matrix}
\right)\\
&=\left(
\begin{matrix}
a\,{\left( 1 + 2\,a\,m\,s \right) }^{-1 + k - n} & 1 \\
 1    & 0
\end{matrix}
\right) \left(
\begin{matrix}
2\,a\,m\,{\left( 1 + 2\,a\,m\,s \right) }^n & 1 \\
 1    & 0
\end{matrix}
\right).
\end{align*}
} Thus
\[
\vec{N}_n = a\,{\left( 1 + 2\,a\,m\,s \right) }^{-1 + k - n},
  2\,a\,m\,{\left( 1 + 2\,a\,m\,s \right) }^n.\]
For a non-square integer $D$, let $l(D)$ denote the length of the
fundamental period in the regular continued fraction expansion of
$\sqrt{D}$. We have proved the following theorem.
\begin{theorem}\label{t1}
Let $a$, $m$, $s$ and $k$ be positive integers. Set
\[
D:=m\,\left( 2\,{\left( 1 + 2\,a\,m\,s \right) }^k +
    m\,{\left( s + a\,{\left( 1 + 2\,a\,m\,s \right) }^k
         \right) }^2 \right).
\]
Then $l(D) =4k+2$ and {\allowdisplaybreaks
\begin{align*}
\sqrt{D}=\bigg[m\,( s + a\,{\left( 1 + 2\,a\,m\,s \right) }^k
);&\overline{\, a,\,
  2\,a\,m\,{\left( 1 + 2\,a\,m\,s \right) }^{k-1},}\\
  &\overline{a(2ams+1),\,
  2am{(2ams+1)}^{k-2},}\\&
\overline{ a(2ams+1)^2,\,
  2am{(2ams+1)}^{k-3},}
\\
&\phantom{sadasdsadfsa}\dots\\
& \overline{a{(1 + 2ams)}^{k-1},\,2am,}\\
&\overline{s+a{(2ams+1)}^k,}\\
&\overline{2am,\,a{(1 +
2ams)}^{k-1},}\\
&\phantom{sadasdsadfsa}\dots,\\
&\overline{2am{(2ams+1)}^{k-3},\,a(2ams+1)^2,}\\
&\overline{2am{(2ams+1)}^{k-2},\,a(2ams+1),}\\
&\overline{2\,a\,m\,{\left( 1 + 2\,a\,m\,s \right) }^{k-1},\,a,}\\
&\overline{2m\,( s + a\,{\left( 1 + 2\,a\,m\,s \right) }^k)}\, \bigg
].
\end{align*}
}
\end{theorem}

This long continued fraction generalizes Madden's first example in
Section 3 of \cite{M01}, where Madden's continued fraction is the
case $m=1$ of the continued fraction above (This continued
fraction of Madden is also given in row 1 of Table 3 in Williams
paper \cite{W85}). The case $m=1$, $a=1$ gives Bernstein's Theorem
3 from \cite{B76}.

\begin{theorem}\label{t1a}
Let $a>1$, $m$, $s$ and $k$ be positive integers. Set
\[
D:=m\,\left( 2\,{\left(2 a\,m\,s-1 \right) }^k +
    m\,{\left( s - a\,{\left(2 a\,m\,s-1 \right) }^k
         \right) }^2 \right).
\]
Then $l(D) =6k+2$ and {\allowdisplaybreaks\begin{align*}
\sqrt{D}=\bigg[m\,( -s &+ a\,{\left(2 a\,m\,s-1 \right) }^k
);\overline{\, a-1,\,1,\,
  2\,a\,m\,{\left(2 a\,m\,s-1 \right) }^{k-1}-1,}\\
  &\overline{a(2ams-1)-1,\,1,\,
  2am{(2ams-1)}^{k-2}-1,}\\&
\overline{ a(2ams-1)^2-1,\,1,\,
  2am{(2ams-1)}^{k-3}-1,}
\\
&\phantom{sadasdsadfsa}\dots\\
& \overline{a{(-1 + 2ams)}^{k-1}-1,\,1,\,2am-1,}\\
&\overline{-s+a{(2ams-1)}^k,}\\
&\overline{2am-1,\,1,\,a{(1 +
2ams)}^{k-1}-1,}\\
&\phantom{sadasdsadfsa}\dots,\\
&\overline{2am{(2ams-1)}^{k-3}-1,\,1,\,a(2ams-1)^2-1,}\\
&\overline{2am{(2ams-1)}^{k-2}-1,\,1,\,a(1+2ams)-1,}\\
&\overline{2\,a\,m\,{\left(2 a\,m\,s-1 \right) }^{k-1}-1,\,1,\,a-1,}\\
&\overline{2m\,( -s + a\,{\left(2 a\,m\,s-1 \right) }^k)}\bigg ].
\end{align*}
}
\end{theorem}
\begin{proof}
In Proposition \ref{p1}, let {\allowdisplaybreaks
\begin{align*}
u&=2 a^2 m (2 a m s-1)^{k-1}-1,\\
v&=a, \\
w&=2 a m,\\
r&=2 a m s-1,\\ x&=\frac{a (2 a m
   s-1)^k-s}{2 a}.
\end{align*}
} With these values, {\allowdisplaybreaks
\begin{align*}
\frac{w r^k}{v}& + w^2 x^2=m\,\left( 2\,{\left(2 a\,m\,s-1 \right)
}^k +
    m\,{\left( s - a\,{\left(2 a\,m\,s-1 \right) }^k
         \right) }^2 \right),\\
N_n&=\left(
\begin{matrix}
u & r^{k-1-n}v \\
 w r^n    & r u -2 v w x
\end{matrix}
\right)\\&= \left(
\begin{matrix}
-1 + 2\,a^2\,m\,{\left(2 a\,m\,s-1 \right) }^{-1 + k}
& a\,{\left(2 a\,m\,s-1 \right) }^{-1 + k - n} \\
 2\,a\,m\,{\left(2 a\,m\,s-1 \right) }^n   &1
\end{matrix}
\right)\\
&=\left(
\begin{matrix}
a\,{\left( 1 + 2ams \right) }^{ k - n-1}-1 & 1 \\
 1    & 0
\end{matrix}
\right)\left(
\begin{matrix}
1 & 1 \\
1    & 0
\end{matrix}
\right) \left(
\begin{matrix}
2\,a\,m\,{\left( 1 + 2ams \right) }^n-1 & 1 \\
 1    & 0
\end{matrix}
\right).
\end{align*}
}
\end{proof}

This example generalizes Madden's second example (page 130
\cite{M01}) (set $m=1$) (Madden's second example can also be found
in row 2 of Table 3 in \cite{W85}), Theorem 8.1 of Levesque and
Rhin \cite{LR86} (set $a=2$, $m=a$ and $s=1$) and Theorem 10.1 of
Levesque and Rhin \cite{LR86} (set $a=2$, $m=a/2$ and $s=2$).

The example given by van der Poorten in \cite{VDP94}, namely
\[
D=\left (q a^k- \frac{a+1}{4 q} \right )^2+a^k \vspace{10pt} \text
{ with }\vspace{10pt} 4 q |a+1,
\]
follows upon setting $a=2q$, $m=1/2$ and $s=(a+1)/(2 q)$.

 We can also let $a$, $m$ and $s$ take negative values in
 Theorem \ref{t1}. Any zero- and negative partial quotients in the
continued fraction expansion can be removed by the following
transformations:
\begin{align}\label{rm0-}
[m,n,0,p,
\alpha]&=[m,n+p, \alpha],\\
[m,-n,\alpha]&=[m-1,1,n-1,-\alpha]. \notag
\end{align}
Out of the eight possible sign combinations for $a$, $m$ and $s$,
only four lead to distinct polynomials, those in Theorems \ref{t1}
and \ref{t1a} (Theorem \ref{t1a} could also have been proved by
replacing $s$ by $-s$ for the case $k$ is even, and replacing $m$
by $-m$ in the case $k$ is odd) and two others.

Our first application of this transformation is to Theorem
\ref{t1a}.
\begin{corollary}
Let $m$, $s$ and $k$ be positive integers such that $ms>1$. Set
\[
D:=m\,\left( 2\,{\left(2 \,m\,s-1 \right) }^k +
    m\,{\left( s - \,{\left(2 \,m\,s-1 \right) }^k
         \right) }^2 \right).
\]
Then $l(D) =6k-2$ and {\allowdisplaybreaks\begin{align*}
\sqrt{D}=\bigg[m\,( -s &+ \,{\left(2 \,m\,s-1 \right) }^k
)+1;\overline{\,
  2\,\,m\,{\left(2 \,m\,s-1 \right) }^{k-1}-1,}\\
  &\overline{(2ms-1)-1,\,1,\,
  2m{(2ms-1)}^{k-2}-1,}\\&
\overline{ (2ms-1)^2-1,\,1,\,
  2m{(2ms-1)}^{k-3}-1,}
\\
&\phantom{sadasdsadfsa}\dots\\
& \overline{{(-1 + 2ms)}^{k-1}-1,\,1,\,2m-1,}\\
&\overline{-s+{(2ms-1)}^k,}\\
&\overline{2m-1,\,1,\,{(1 +
2ms)}^{k-1}-1,}\\
&\phantom{sadasdsadfsa}\dots,\\
&\overline{2m{(2ms-1)}^{k-3}-1,\,1,\,(2ms-1)^2-1,}\\
&\overline{2m{(2ms-1)}^{k-2}-1,\,1,\,(1+2ms)-1,}\\
&\overline{2\,\,m\,{\left(2 \,m\,s-1 \right) }^{k-1}-1,}
\overline{2m\,( -s + \,{\left(2 \,m\,s-1 \right) }^k)+2}\,\bigg ].
\end{align*}
}
\end{corollary}
\begin{proof}
This follows upon letting $a=1$ in Theorem \ref{t1a} and using
\eqref{rm0-} to remove the zeros resulting from the ``$a-1$"
terms.
\end{proof}

This continued fraction generalizes that in Theorem 2 of Bernstein
\cite{B76} (set $m=1$ and $s=a+1$) and also that in Theorem 6.1 of
Levesque and Rhin \cite{LR86} (set $m=2a$ and $s=1$).

\begin{theorem}\label{t1b}
Let $a>1$, $m$, $s$ and $k$ be positive integers. Set
\[
D:=m\,\left(
     m\,{\left( s + a\,{\left( 2\,a\,m\,s -1 \right) }^{k}
          \right) }^2- 2\,{\left( 2\,a\,m\,s -1 \right) }^{k} \right).
\]
Then  $l(D) =6k+4$ and {\allowdisplaybreaks
\begin{align}\label{c1eq}
\sqrt{D}=\bigg[m\,( s &+ a\,{\left(2 a\,m\,s-1 \right) }^{k}
)-1;\overline{1,\, a-1,\,
  2\,a\,m\,{\left(2 a\,m\,s-1 \right) }^{k-1}-1,}\\
  &\overline{1,\,a(2ams-1)-1,\,
  2am{(2ams-1)}^{k-2}-1,}  \notag\\&
\overline{1,\, a(2ams-1)^2-1,\,
  2am{(2ams-1)}^{k-3}-1,}\notag
\\
&\phantom{sadasdsadfsa}\dots  \notag \\
& \overline{1,\,a{( 2ams-1)}^{k-1}-1,\,2am-1,}  \notag \\
&\overline{1,\,s+a{(2ams-1)}^{k}-2,\,1,}  \notag \\
&\overline{2am-1,\,a{(2ams-1)}^{k-1}-1,\,1,}  \notag   \\
&\phantom{sadasdsadfsa}\dots,  \notag \\
&\overline{2am{(2ams-1)}^{k-3}-1,\,a(2ams-1)^2-1,\,1,}  \notag \\
&\overline{2am{(2ams-1)}^{k-2}-1,\,a(2ams-1)-1,\,1,}  \notag \\
&\overline{2\,a\,m\,{\left(2 a\,m\,s-1 \right) }^{k-1}-1,\,a-1,\,1,}  \notag \\
&\overline{2m\,( s + a\,{\left(2 a\,m\,s-1 \right)
}^{k})-2}\,\bigg ].  \notag
\end{align}
}
\end{theorem}
\begin{proof}
We consider the cases where $k$ is odd and $k$ is even separately.
We first consider $k$ odd.

Replace $a$ by $-a$ and $k$ by $2k+1$ in Theorem \ref{t1}. Then
\[
D:=m\,\left(
    m\,{\left( s + a\,{\left(  2\,a\,m\,s -1\right) }^{2k+1}
         \right) }^2 - 2\,{\left(  2\,a\,m\,s -1 \right) }^{2k+1} \right)
\]
and {\allowdisplaybreaks
\begin{align*}
\sqrt{D}=\bigg[m\,( s + a\,{\left( 2\,a\,m\,s -1\right) }^{2k+1}
);&\overline{\, -a,\,
 - 2\,a\,m\,{\left( 2\,a\,m\,s-1 \right) }^{2k},}\\
  &\overline{a(2ams-1),\,
  2am{(2ams-1)}^{2k-1},}\\&
\overline{ -a(2ams-1)^2,\,
  -2am{(2ams-1)}^{2k-2},}
\\
&\phantom{sadasdsadfsa}\dots\\
& \overline{a{(2ams-1)}^{2k-1},\,2am(2ams-1),}\\
& \overline{-a{( 2ams-1)}^{2k},\,-2am,}\\
&\overline{s+a{(2ams-1)}^{2k+1},}\\
&\overline{-2am,\,-a{(2ams-1)}^{2k},}\\
&\overline{2am(2ams-1),\,a{(2ams-1)}^{2k-1},}\\
&\phantom{sadasdsadfsa}\dots,\\
&\overline{-2am{(2ams-1)}^{2k-2},\,-a(2ams-1)^2,}\\
&\overline{2am{(2ams-1)}^{2k-1},\,a(2ams-1),}\\
&\overline{-2\,a\,m\,{\left(  2\,a\,m\,s -1\right) }^{2k},\,-a,}\\
&\overline{2m\,( s + a\,{\left( 2\,a\,m\,s-1 \right) }^{2k+1})}\,
\bigg ].
\end{align*}}
We now apply the second identity at \eqref{rm0-} repeatedly to get
that {\allowdisplaybreaks\begin{align*} \sqrt{D}=\bigg[m\,( s &+
a\,{\left(2 a\,m\,s-1 \right) }^{2k+1} )-1;\overline{1,\, a-1,\,
  2\,a\,m\,{\left(2 a\,m\,s-1 \right) }^{2k}-1,}\\
  &\overline{1,\,a(2ams-1)-1,\,
  2am{(2ams-1)}^{2k-1}-1,}  \notag\\&
\overline{1,\, a(2ams-1)^2-1,\,
  2am{(2ams-1)}^{2k-2}-1,}\notag
\\
&\phantom{sadasdsadfsa}\dots  \notag \\
& \overline{1,\,a{( 2ams-1)}^{2k}-1,\,2am-1,}  \notag \\
&\overline{1,\,s+a{(2ams-1)}^{2k+1}-2,\,1,}  \notag \\
&\overline{1,\,2am-1,\,a{(2ams-1)}^{2k}-1,}  \notag   \\
&\phantom{sadasdsadfsa}\dots,  \notag \\
&\overline{1,\,2am{(2ams-1)}^{2k-2}-1,\,a(2ams-1)^2-1,}  \notag \\
&\overline{1,\,2am{(2ams-1)}^{2k-1}-1,\,a(2ams-1)-1,}  \notag \\
&\overline{1,\,2\,a\,m\,{\left(2 a\,m\,s-1 \right) }^{2k}-1,\,a-1,}  \notag \\
&\overline{2m\,( s + a\,{\left(2 a\,m\,s-1 \right)
}^{2k+1})-2}\,\bigg ].  \notag
\end{align*}
} It is clear that   $l(D) =12k+10$. This proves the theorem for
odd $k$.

We next consider the case where $k$ is even. Replace $a$ by $-a$,
$s$ by $-s$, $m$ by $-m$ and $k$ by $2k$ in Theorem \ref{t1}. Then
\[
D:=m\,\left(
    m\,{\left( s + a\,{\left(  2\,a\,m\,s -1\right) }^{2k}
         \right) }^2 - 2\,{\left(  2\,a\,m\,s -1 \right) }^{2k} \right)
\]
and {\allowdisplaybreaks
\begin{align*}
\sqrt{D}=\bigg[m\,( s + a\,{\left( 2\,a\,m\,s -1\right) }^{2k}
);&\overline{\, -a,\,
 - 2\,a\,m\,{\left( 2\,a\,m\,s-1 \right) }^{2k-1},}\\
  &\overline{a(2ams-1),\,
  2am{(2ams-1)}^{2k-2},}\\&
\overline{ -a(2ams-1)^2,\,
  -2am{(2ams-1)}^{2k-3},}
\\
&\phantom{sadasdsadfsa}\dots\\
& \overline{a{(2ams-1)}^{2k-2},\,2am(2ams-1),}\\
& \overline{-a{( 2ams-1)}^{2k-1},\,-2am,}\\
&\overline{s+a{(2ams-1)}^{2k},}\\
&\overline{-2am,\,-a{(2ams-1)}^{2k-1},}\\
&\overline{2am(2ams-1),\,a{(2ams-1)}^{2k-2},}\\
&\phantom{sadasdsadfsa}\dots,\\
&\overline{-2am{(2ams-1)}^{2k-3},\,-a(2ams-1)^2,}\\
&\overline{2am{(2ams-1)}^{2k-2},\,a(2ams-1),}\\
&\overline{-2\,a\,m\,{\left(  2\,a\,m\,s -1\right) }^{2k-1},\,-a,}\\
&\overline{2m\,( s + a\,{\left( 2\,a\,m\,s-1 \right) }^{2k})}\,
\bigg ].
\end{align*}}
We again apply the second identity at \eqref{rm0-} repeatedly and
the stated continued fraction expansion follows. It is clear that
$l(D) =12k+4$ in this case. This completes the proof for even $k$.
\end{proof}

This theorem generalizes Theorem 7.1 of Levesque and Rhin
\cite{LR86} (set $a=2$, $s=1$ and $m=a$) and Theorem 9.1 of
Levesque and Rhin \cite{LR86} (set $a=2$, $s=2$ and $m=a/2$).
Williams example in row 3 of Table 3 in \cite{W85} is the case
$m=1$ of the continued fraction above.

Remark: It seems likely from the common form of the expansions that
there should be an alternative proof of Theorem \ref{t1b} that
covers the even and odd cases simultaneously. However, we do not
pursue that here.

\begin{corollary}
Let  $m$, $s$ and $k$ be positive integers. Set
\[
D:=m\,\left(
     m\,{\left( s + \,{\left( 2\,m\,s -1 \right) }^{k}
          \right) }^2- 2\,{\left( 2\,m\,s -1 \right) }^{k} \right).
\]
Then  $l(D) =6k$ and {\allowdisplaybreaks
\begin{align}\label{cc1eq}
\sqrt{D}=\bigg[m\,( s &+ {\left(2 m\,s-1 \right) }^{k} )-1;\,
 \overline{ 2\,m\,{\left(2 m\,s-1 \right) }^{k-1}},\\
  &\overline{1,\,(2ms-1)-1,\,
  2m{(2ms-1)}^{k-2}-1,}  \notag\\&
\overline{1,\, (2ms-1)^2-1,\,
  2m{(2ms-1)}^{k-3}-1,}\notag
\\
&\phantom{sadasdsadfsa}\dots  \notag \\
& \overline{1,\,{( 2ms-1)}^{k-1}-1,\,2m-1,}  \notag \\
&\overline{1,\,s+{(2ms-1)}^{k}-2,\,1,}  \notag \\
&\overline{2m-1,\,{(2ms-1)}^{k-1}-1,\,1,}  \notag   \\
&\phantom{sadasdsadfsa}\dots,  \notag \\
&\overline{2m{(2ms-1)}^{k-3}-1,\,(2ms-1)^2-1,\,1,}  \notag \\
&\overline{2m{(2ms-1)}^{k-2}-1,\,(2ms-1)-1,\,1,}  \notag \\
&\overline{2\,m\,{\left(2 m\,s-1 \right) }^{k-1},\,  2m\,( s +
\,{\left(2 m\,s-1 \right) }^{k})-2}\,\bigg ]. \notag
\end{align}
}
\end{corollary}
\begin{proof}
Let $a=1$ in Theorem \ref{t1b} and use \eqref{rm0-} to remove the
zeros resulting from the ``$a-1$" terms.
\end{proof}

This result generalizes that in Theorem 1 of Bernstein \cite{B76}
(set $s=a+1$ and $m=1$) and also that in Theorem 5.1 of Levesque
and Rhin \cite{LR86} (set $s=1$ and $m=2a$). It also generalizes
the example in row 5 of Table 3 in \cite{W85}.

\begin{theorem}\label{t1c}
Let $a>2$, $m$, $s$ and $k$ be positive integers. Set
\[
D:=m\,\left(
    m\,{\left( -s + a\,{\left( 1 + 2\,a\,m\,s \right) }^k
         \right) }^2 - 2\,{\left( 1 + 2\,a\,m\,s \right) }^k \right).
\]
Then $l(D) =8k+4$ and {\allowdisplaybreaks
\begin{align*}
\sqrt{D}=\bigg[m\,( -s +& a\,{\left( 1 + 2\,a\,m\,s \right) }^k
)-1;\overline{1,\, a-2,\,1,
  2\,a\,m\,{\left( 1 + 2\,a\,m\,s \right) }^{k-1}-2,}\\
  &\overline{1,a(2ams+1)-2,\,1,
  2am{(2ams+1)}^{k-2}-2,}\\&
\overline{1, a(2ams+1)^2-2,\,1,
  2am{(2ams+1)}^{k-3}-2,}
\\
&\phantom{sadasdsadfsa}\dots\\
& \overline{1,a{(1 + 2ams)}^{k-1}-2,\,1,2am-2,}\\
&\overline{1,-s+a{(2ams+1)}^k-2,1,}\\
&\overline{2am-2,\,1,a{(1 +
2ams)}^{k-1}-2,1,}\\
&\phantom{sadasdsadfsa}\dots,\\
&\overline{2am{(2ams+1)}^{k-3}-2,\,1,a(2ams+1)^2-2,1,}\\
&\overline{2am{(2ams+1)}^{k-2}-2,\,1,a(2ams+1)-2,1,}\\
&\overline{2am\,{\left( 1 + 2\,a\,m\,s \right) }^{k-1}-2,\,1,a-2,1,}\\
&\overline{2m\,( -s + a\,{\left( 1 + 2\,a\,m\,s \right) }^k)-2}\,
\bigg ].
\end{align*}
}
\end{theorem}

\begin{proof}
Replace $a$ by $-a$ and $m$ by $-m$ in Theorem \ref{t1}. Then
\[
D:=m\,\left(
    m\,{\left(- s + a\,{\left(  2\,a\,m\,s +1\right) }^{k}
         \right) }^2 - 2\,{\left(  2\,a\,m\,s +1 \right) }^{k} \right)
\]
and {\allowdisplaybreaks
\begin{align*}
\sqrt{D}=\bigg[m\,( -s + a\,{\left( 2\,a\,m\,s +1\right) }^{k}
);&\overline{\, -a,\,
 2\,a\,m\,{\left( 2\,a\,m\,s+1 \right) }^{k-1},}\\
  &\overline{-a(2ams+1),\,
  2am{(2ams+1)}^{k-2},}\\&
\overline{ -a(2ams+1)^2,\,
  2am{(2ams+1)}^{k-3},}
\\
&\phantom{sadasdsadfsa}\dots\\
& \overline{-a{(2ams-1)}^{k-2},\,2am(2ams+1),}\\
& \overline{-a{( 2ams+1)}^{k-1},\,2am,}\\
&\overline{s-a{(2ams+1)}^{k},}\\
&\overline{2am,\,-a{(2ams+1)}^{k-1},}\\
&\overline{2am(2ams+1),\,-a{(2ams+1)}^{k-2},}\\
&\phantom{sadasdsadfsa}\dots,\\
&\overline{2am{(2ams+1)}^{k-3},\,-a(2ams+1)^2,}\\
&\overline{2am{(2ams+1)}^{k-2},\,-a(2ams+1),}\\
&\overline{2\,a\,m\,{\left(  2\,a\,m\,s -1\right) }^{k-1},\,-a,}\\
&\overline{2m\,( -s + a\,{\left( 2\,a\,m\,s+1 \right) }^{k})}\,
\bigg ].
\end{align*}}
We again apply the second identity at \eqref{rm0-} repeatedly and
the stated continued fraction expansion follows. It is clear that
$l(D) =8k+4$ in this case. This completes the proof.
\end{proof}

Williams continued fraction in row 9 of Table 3 in \cite{W85} is
the case $m=1$ of the theorem above.

\begin{corollary}\label{c1}
Let  $m$, $s$ and $k>1$ be positive integers such that $ms>1$. Set
\[
D:=m\,\left(
    m\,{\left( -s + \,{\left( 1 + 2\,m\,s \right) }^k
         \right) }^2 - 2\,{\left( 1 + 2\,m\,s \right) }^k \right).
\]
Then $l(D) =8k$ and {\allowdisplaybreaks
\begin{align*}
\sqrt{D}=\bigg[m\,( -s +& {\left( 1 + 2\,m\,s \right) }^k
)-2;\overline{1,\,
  2\,m\,{\left( 1 + 2\,m\,s \right) }^{k-1}-3,}\\
  &\overline{1,(2ms+1)-2,\,1,
  2m{(2ms+1)}^{k-2}-2,}\\&
\overline{1, (2ms+1)^2-2,\,1,
  2m{(2ms+1)}^{k-3}-2,}
\\
&\phantom{sadasdsadfsa}\dots\\
& \overline{1,{(1 + 2ms)}^{k-1}-2,\,1,2m-2,}\\
&\overline{1,-s+{(2ms+1)}^k-2,1,}\\
&\overline{2m-2,\,1,{(1 +
2ms)}^{k-1}-2,1,}\\
&\phantom{sadasdsadfsa}\dots,\\
&\overline{2m{(2ms+1)}^{k-3}-2,\,1,(2ms+1)^2-2,1,}\\
&\overline{2m{(2ms+1)}^{k-2}-2,\,1,(2ms+1)-2,1,}\\
&\overline{2m\,{\left( 1 + 2\,m\,s \right) }^{k-1}-3,\,1,}\\
&\overline{2m\,( -s +{\left( 1 + 2\,m\,s \right) }^k)-4}\, \bigg
].
\end{align*}
}
\end{corollary}
\begin{proof}
This follows after letting $a=1$ in Theorem \ref{t1c} and using
\eqref{rm0-} to remove the resulting zeroes and negatives.
\end{proof}

Williams continued fraction in row 8 of Table 3 in \cite{W85} is
the case $m=1$ of the corollary above.

\begin{corollary}\label{c2}
Let  $k>2$ and $s$  be positive integers. Set
\[
D:=
    {\left(  \,{\left( 1 + 2\,s \right) }^k-s
         \right) }^2 - 2\,{\left( 1 + 2\,s \right) }^k.
\]
Then $l(D) =8k-4$ and {\allowdisplaybreaks
\begin{align*}
\sqrt{D}=\bigg[& {\left( 1 + 2\,s \right) }^k -s-2;\overline{1,\,
  2\,{\left( 1 + 2\,s \right) }^{k-1}-3,}\\
  &\overline{1,(2s+1)-2,\,1,
  2{(2s+1)}^{k-2}-2,}\\&
\overline{1, (2s+1)^2-2,\,1,
  2{(2s+1)}^{k-3}-2,}
\\
&\phantom{sadasdsadfsa}\dots\\
& \overline{1,{(1 + 2s)}^{k-2}-2,\,1,\,2(2s+1)-2}\\
& \overline{1,{(1 + 2s)}^{k-1}-2,\,2,}\\
&\overline{-s+{(2s+1)}^k-2,}\\
&\overline{2,{(1 +
2s)}^{k-1}-2,1,}\\
&\overline{2(2s+1)-2,\,1,\,,{(1 +
2s)}^{k-1}-2,1,}\\
&\phantom{sadasdsadfsa}\dots,\\
&\overline{2{(2s+1)}^{k-3}-2,\,1,(2s+1)^2-2,1,}\\
&\overline{2{(2s+1)}^{k-2}-2,\,1,(2s+1)-2,1,}\\
&\overline{2{\left( 1 + 2\,s \right) }^{k-1}-3,\,1,}\\
&\overline{2( s +{\left( 1 + 2\,s \right) }^k)-4}\, \bigg ].
\end{align*}
}
\end{corollary}
\begin{proof}
This follows after letting $m=1$ in Corollary \ref{c1} and using
\eqref{rm0-} to remove the zero resulting from the ``$2m-2$"
terms.
\end{proof}
The result above is Theorem 4 in Bernstein \cite{B76}.

\begin{corollary}\label{c3}
Let $m$, $s$ and $k>1$ be positive integers. Set
\[
D:=m\,\left(
    m\,{\left( -s + 2\,{\left( 1 + 4m\,s \right) }^k
         \right) }^2 - 2\,{\left( 1 + 4m\,s \right) }^k \right).
\]
Then $l(D) =8k$ and {\allowdisplaybreaks
\begin{align*}
\sqrt{D}=\bigg[m\,( -s +& 2\,{\left( 1 + 4m\,s \right) }^k
)-1;\overline{ 2,
  4m\,{\left( 1 + 4m\,s \right) }^{k-1}-2,}\\
  &\overline{1,2(4ms+1)-2,\,1,
  4m{(4ms+1)}^{k-2}-2,}\\&
\overline{1, 2(4ms+1)^2-2,\,1,
  4m{(4ms+1)}^{k-3}-2,}
\\
&\phantom{sadasdsadfsa}\dots\\
& \overline{1,2{(1 + 4ms)}^{k-1}-2,\,1,4m-2,}\\
&\overline{1,-s+2{(4ms+1)}^k-2,1,}\\
&\overline{4m-2,\,1,2{(1 +
4ms)}^{k-1}-2,1,}\\
&\phantom{sadasdsadfsa}\dots,\\
&\overline{4m{(4ms+1)}^{k-3}-2,\,1,2(4ms+1)^2-2,1,}\\
&\overline{4m{(4ms+1)}^{k-2}-2,\,1,2(4ms+1)-2,1,}\\
&\overline{4m\,{\left( 1 + 4m\,s \right) }^{k-1}-2,\,2,}\\
&\overline{2m\,( -s + 2\,{\left( 1 + 4m\,s \right) }^k)-2}\, \bigg
].
\end{align*}
}
\end{corollary}
\begin{proof}
This follows after letting $a=2$ in Theorem \ref{t1c} and using
\eqref{rm0-} to remove the zero resulting from the ``$a-2$" terms.
\end{proof}

We next use Propositions \ref{p2} and \ref{p3} to construct
families of long continued fractions with no central partial
quotient in the fundamental period.
\begin{theorem}\label{t2}
Let $b$,  $s$ and $k$ be positive integers. Set
\[
D:=(4 b s+1)^{k}+\left(b (4 b s+1)^{k}+s\right)^2.
\]
Then $l(D) =2k+1$ and {\allowdisplaybreaks
\begin{align*}
\sqrt{D}=\bigg[ b (4 b s+1)^{ k}+s;\,\, &\overline{ 2 b,\,
\hspace{10pt}2 b (4 b
s+1)^{ k-1},}\\
&\overline{2 b (4 b s+1),\,2 b (4 b s+1)^{ k-2} ,}\\
& \overline{ 2 b (4 b s+1)^2,\,2 b (4 b s+1)^{ k-3},}\\
&\phantom{sadasdsadfsa}\dots\\
&\overline{2 b (4 b s+1)^{\lfloor (k-1)/2\rfloor}, \,2 b (4 b s+1)^{k-1-\lfloor (k-1)/2\rfloor},}\\
&\phantom{sadasdsadfsa}\dots,\\
&\overline{2 b (4 b s+1)^{ k-3},\, 2 b (4 b s+1)^2,}\\
&\overline{2 b (4 b s+1)^{ k-2},\,2 b (4 b s+1),}\\
&\overline{ 2 b (4 b s+1)^{ k-1},\,2 b,\,\,}\overline{ 2( b (4 b
s+1)^{ k}+s)}\bigg ].
\end{align*}
} Here we understand the mid-point of the period to be just after
the $2 b (4 b s+1)^{\lfloor (k-1)/2\rfloor}$ term, if $k$ is odd,
and just after the $2 b (4 b s+1)^{k-1-\lfloor (k-1)/2\rfloor}$
term, if $k$ is even.
\end{theorem}
\begin{proof}
We first prove this for even $k$.  In Proposition \ref{p2} set
$v=2b$, $r=1+4 b s$, $u=1+4 b^2(1+4 b s)^{2k-1}$ and $x= s+b(1+4 b
s)^{2k}$. Then
\[
D=(4 b s+1)^{2k}+\left(b (4 b s+1)^{2k}+s\right)^2.
\]
With these values also,
\[
N_n =\left(
\begin{matrix}
u & r^{k-1-n}v \\
r^{k+n}v     & r u -2 v  x
\end{matrix}
\right) =\left(
\begin{matrix}
2b(1+4bs)^{k-1-n} & 1 \\
1    & 0
\end{matrix}
\right)\left(
\begin{matrix}
2b(1+4bs)^{k+n} & 1 \\
1     & 0
\end{matrix}
\right),
\]
and so $\vec{N}_{n}=2b(1+4bs)^{k-1-n},\,2b(1+4bs)^{k+n}$. Thus
{\allowdisplaybreaks
\begin{align*}
\sqrt{D}=\bigg[ b (4 b s+1)^{2 k}+s;\,\, &\overline{ 2 b,\,
\hspace{10pt}2 b (4 b
s+1)^{2 k-1},}\\
&\overline{2 b (4 b s+1),\,2 b (4 b s+1)^{2 k-2} ,}\\
& \overline{ 2 b (4 b s+1)^2,\,2 b (4 b s+1)^{2 k-3},}\\
&\phantom{sadasdsadfsa}\dots\\
&
\overline{ 2 b (4 b s+1)^{k-1},\,2 b (4 b s+1)^k,}\\
&\overline{2 b (4 b s+1)^k, \,2 b (4 b s+1)^{k-1},}\\
&\phantom{sadasdsadfsa}\dots,\\
&\overline{2 b (4 b s+1)^{2 k-3},\, 2 b (4 b s+1)^2,}\\
&\overline{2 b (4 b s+1)^{2 k-2},\,2 b (4 b s+1),}\\
&\overline{ 2 b (4 b s+1)^{2 k-1},\,2 b,\,\,}\overline{ 2( b (4 b
s+1)^{2 k}+s)}\bigg ],
\end{align*}
} and the result  follows for even $k$.

For odd $k$ we use Proposition \ref{p3} above, where we set
$v=2b$, $r=1+4 b s$, $u=1+4 b^2(1+4 b s)^{2k}$, $q=b(1+4 b s)^{k}$
and $w= s$. Then
\[
D=(4 b s+1)^{2k+1}+\left(b (4 b s+1)^{2k+1}+s\right)^2.
\]
With these values also,
\begin{align*}
N_n &=\left(
\begin{matrix}
u & r^{k-1-n}v \\
r^{n}v\,\left( r^k + 4\,q\,w \right)    & r\,u - 2\,q\,r^k\,v -
    2\,\left( 1 + 4\,q^2 \right) \,v\,w
\end{matrix}
\right)\\&=\left(
\begin{matrix}
2b(1+4bs)^{k-1-n} & 1 \\
1    & 0
\end{matrix}
\right)\left(
\begin{matrix}
2b(1+4bs)^{k+n+1} & 1 \\
1     & 0
\end{matrix}
\right),
\end{align*}
and so $\vec{N}_{n}=2b(1+4bs)^{k-1-n},\,2b(1+4bs)^{k+n+1}$. Thus
{\allowdisplaybreaks
\begin{align*} \sqrt{D}=\bigg[ b (4 b s+1)^{2
k+1}+s;\,\, &\overline{ 2 b,\, \hspace{10pt}2 b (4 b
s+1)^{2 k},}\\
&\overline{2 b (4 b s+1),\,2 b (4 b s+1)^{2 k-1} ,}\\
& \overline{ 2 b (4 b s+1)^2,\,2 b (4 b s+1)^{2 k-2},}\\
&\phantom{sadasdsadfsa}\dots\\
&\overline{ 2 b (4 b s+1)^{k-1},\,2 b (4 b s+1)^{k+1},}\\
&\overline{2 b (4 b s+1)^k, \,2 b (4 b s+1)^{k},}\\
&\overline{2 b (4 b s+1)^{k+1}, \,2 b (4 b s+1)^{k-1},}\\
&\phantom{sadasdsadfsa}\dots,\\
&\overline{2 b (4 b s+1)^{2 k-2},\, 2 b (4 b s+1)^2,}\\
&\overline{2 b (4 b s+1)^{2 k-1},\,2 b (4 b s+1),}\\
&\overline{ 2 b (4 b s+1)^{2 k},\,\,2 b,\,\,}\overline{ 2( b (4 b
s+1)^{2 k+1}+s)}\bigg ],
\end{align*}
} and the result again follows.
\end{proof}

The result above (without the explicit continued fraction
expansion) appears in row 1 of Table 2 in \cite{W85}.

As previously, we can let $b$ or $r$ take negative integral values
and produce a new long continued fraction.
\begin{corollary}\label{c6}
Let $b$,  $s$ and $k$ be positive integers. Set
\[
D:=(4 b s-1)^{2k}+\left(b (4 b s-1)^{2k}-s\right)^2.
\]
Then $l(D) =6k+1$ and {\allowdisplaybreaks
\begin{align*}
\sqrt{D}=\bigg[ b (4 b s-1)^{2 k}-s;&\overline{ 2 b-1,\,1, 2 b (4 b
s-1)^{2 k-1}-1,}\\
&\overline{2 b (4 b s-1)-1,\,1,2 b (4 b s-1)^{2 k-2}-1 ,}\\
& \overline{ 2 b (4 b s-1)^2-1,\,1,2 b (4 b s-1)^{2 k-3}-1,}\\
&\phantom{sadasdsadfsa}\dots\\
&
\overline{ 2 b (4 b s-1)^{k-1}-1,\,1,2 b (4 b s-1)^k-1,}\\
&\overline{2 b (4 b s-1)^k-1, \,1,2 b (4 b s-1)^{k-1}-1,}\\
&\phantom{sadasdsadfsa}\dots,\\
&\overline{2 b (4 b s-1)^{2 k-3}-1,\, 2 b (4 b s-1)^2-1,}\\
&\overline{2 b (4 b s-1)^{2 k-2}-1,\,1,2 b (4 b s-1)-1,}\\
&\overline{ 2 b (4 b s-1)^{2 k-1}-1,\,1,2 b-1,\,\,}\overline{ 2( b
(4 b s-1)^{2 k}-s)}\bigg ]
\end{align*}
}
\end{corollary}
\begin{proof}
This follows from Theorem \ref{t2} in the case $k$ is even, after
replacing $s$ by $-s$ and using \eqref{rm0-} to remove the resulting
negative partial quotients from the resulting continued fraction
expansion
\end{proof}

\begin{corollary}\label{c7}
Let $b$,  $s$ and $k$ be positive integers. Set
\[
D:=\left(b (4 b s-1)^{2k+1}+s\right)^2-(4 b s-1)^{2k+1}.
\]
Then $l(D) =6k+5$ and {\allowdisplaybreaks
\begin{align*} \sqrt{D}=\bigg[ b (4 b &s-1)^{2
k+1}+s-1;\,\, \overline{1,\, 2 b-1,\, 2 b (4 b
s-1)^{2 k}-1,}\\
&\overline{1,\,2 b (4 b s-1)-1,\,2 b (4 b s-1)^{2 k-1}-1 ,}\\
& \overline{1,\, 2 b (4 b s-1)^2-1,\,2 b (4 b s-1)^{2 k-2}-1,}\\
&\phantom{sadasdsadfsa}\dots\\
&\overline{1,\, 2 b (4 b s-1)^{k-1}-1,\,2 b (4 b s-1)^{k+1}-1,}\\
&\overline{1,\,2 b (4 b s-1)^k-1, \,2 b (4 b s-1)^{k}-1,1,\,}\\
&\overline{2 b (4 b s-1)^{k+1}-1, \,2 b (4 b s-1)^{k-1}-1,1,\,}\\
&\phantom{sadasdsadfsa}\dots,\\
&\overline{2 b (4 b s-1)^{2 k-2}-1,\, 2 b (4 b s-1)^2-1,1,\,}\\
&\overline{2 b (4 b s-1)^{2 k-1}-1,\,2 b (4 b s-1)-1,1,\,}\\
&\overline{ 2 b (4 b s-1)^{2 k}-1,\,\,2 b-1,\,1,\,}\overline{ 2( b
(4 b s-1)^{2 k+1}+s-2)}\bigg ].
\end{align*}
}
\end{corollary}
\begin{proof}
This follows from Theorem \ref{t2} in the case $k$ is odd, after
replacing $b$ by $-b$ and using \eqref{rm0-} to remove the
resulting negative partial quotients from the resulting continued
fraction expansion.
\end{proof}

Remark: Statements similar to those in Corollaries \ref{c6} and
\ref{c7} hold if $2k$ is replaced by $2k+1$, but these results
cannot be derived from Theorem \ref{t2}. All of these statements,
without explicit continued fraction expansions, appear in Table 2
of \cite{W85}.

\section{Fundamental Units in Real Quadratic Fields}

It is straightforward in many cases to compute the fundamental
units in the real quadratic fields corresponding to the surds in
the various theorems and corollaries. In what follows, we assume
$D$ is square free and $D \not \equiv 5 (\mod{8})$.

In Theorem \ref{t1}, for example,  where
\[D=m(2(1 + 2a m s)^k + m(s + a(1 + 2 a m s)^k)^2),\] it follows
directly from Proposition \ref{p1} that the fundamental unit in
$\mathbb{Q}(\sqrt{D})$ is
\[
\frac{2\,m\,{\left( 1 + a\,\left( {\sqrt{D}} +
           m\,\left( s + a\,{\left( 1 + 2\,a\,m\,s \right) }^k \right)  \right)
        \right) }^{2\,k}}{{\left( {\sqrt{D}} -
       m\,\left( s + a\,{\left( 1 + 2\,a\,m\,s \right) }^k \right)  \right)
       }^2}.
\]
With $m = 3$, $a =5$, $s =7$ and $k = 2$, for example, we get
$D=446005190022$ and that the fundamental unit in
$\mathbb{Q}(\sqrt{446005190022})$ is (after simplifying)
\begin{multline*}
149199899813252915906267542273 \\+
223407925198820626278032\sqrt{446005190022}.
\end{multline*}

Likewise, in Theorem \ref{t2} (for even $k$), where
\[
D= (1 + 4bs)^{2k} + (s + b(1 + 4bs)^{2k})^2,
\]
Proposition \ref{p2} gives that the fundamental unit in
$\mathbb{Q}(\sqrt{D})$ is {\allowdisplaybreaks
\begin{multline*}
\frac{1}{{\left( 1 + 4bs \right) }^{2k}}\left( s + b{\left( 1 +
4bs \right) }^{2k} +
      {\sqrt{{\left( 1 + 4bs \right) }^{2k} +
          {\left( s + b{\left( 1 + 4bs \right) }^{2k} \right) }^2}} \right)
         \\
   \times   \bigg( -2\,b\,\left( s + b\,{\left( 1 + 4\,b\,s \right) }^{2\,k} \right)  +
        \left( 1 + 4\,b\,s \right) \,
         \left( 1 + 4\,b^2\,{\left( 1 + 4\,b\,s \right) }^{-1 + 2\,k} \right)
         \\
       + 2\,b\,{\sqrt{{\left( 1 + 4\,b\,s \right) }^{2\,k} +
             {\left( s + b\,{\left( 1 + 4\,b\,s \right) }^{2\,k} \right) }^2}} \,\bigg)
        ^{2\,k}.
\end{multline*}
} Setting $s = 3$, $b = 5$, and $k =2$, for example, gives  $D =
4792683254153105$ and that the fundamental unit in
$\mathbb{Q}(\sqrt{4792683254153105})$ is
\begin{multline*}
18375851029288260766491636025698114848 +\\
265434944781468068474213871001
   \sqrt{4792683254153105}.
\end{multline*}

For our last example, we also consider Theorem \ref{t2} (for odd
$k$), where
\[
D= (1 + 4bs)^{2k+1} + (s + b(1 + 4bs)^{2k+1})^2,
\]
Proposition \ref{p3} gives that the fundamental unit in
$\mathbb{Q}(\sqrt{D})$ is {\allowdisplaybreaks
\begin{multline*}
\frac{1}{{\left( 1 + 4\,b\,s \right) }^{2\,k}}
  {\left( 1 + 2\,b\,{\sqrt{D}} + 2\,b\,s +
      2\,b^2\,{\left( 1 + 4\,b\,s \right) }^{1 + 2\,k} \right)
      }^{2\,k}\times\\
   \left( {\sqrt{D}} + s + 4\,b^3\,{\left( 1 + 4\,b\,s \right) }^{1 + 4\,k} +
   b\,{\left( 1 + 4\,b\,s \right) }^{2\,k}\,
    \left( 3 + 4\,b\,\left( {\sqrt{D}} + 2\,s \right)  \right)\right).
\end{multline*}
} Letting $s = 5$, $b = 2$, and $k =1$ gives  $D = 19001864330$
and shows that the fundamental unit in
$\mathbb{Q}(\sqrt{19001864330})$ is
\[
2682318982172034563 + 19458632525153\sqrt{19001864330}.
\]

\section{Concluding Remarks}
While considering the various long continued fractions in the
papers by Bernstein \cite{B76}, \cite{B76a},  Levesque and
 Rhin \cite{LR86} and Williams \cite{W85}, and trying to see if any more
 of the continued fractions in these papers could be generalized
 by the methods in the present paper, we were
 led to a new construction.

 This new construction succeeds where the methods in the present paper fail,
 in that it allowed us to generalize some more of the continued fractions in
 the papers mentioned above.

 We will investigate this new construction in
 a subsequent paper.

 \allowdisplaybreaks{

}
\end{document}